\documentclass[11pt]{amsart}

\usepackage{amsmath}
\usepackage{amsfonts}
\usepackage{amssymb}
\usepackage{amsthm}
\usepackage{epsfig}
\usepackage{graphicx}
\usepackage{url}
\usepackage[colorlinks=true]{hyperref}
\usepackage{nicefrac}

\newtheorem{theorem}{Theorem}[section]
\newtheorem{lemma}{Lemma}[section]
\newtheorem{corollary}{Corollary}[section]

\newtheorem{remark}{Remark}[section]
\newtheorem{proposition}{Proposition}[section]

\DeclareMathOperator{\inter}{int}

\DeclareMathOperator{\iV}{V}
\DeclareMathOperator{\vol}{vol}

\DeclareMathOperator{\LE}{G}

\def\conv{\mathop\mathrm{conv}\nolimits}

\def\K{\mathcal{K}}
\def\P{\mathcal{P}}
\def\L{\mathcal{L}}

\def\Q{\mathcal{Q}}

\def\R{\mathbb{R}}

\def\Z{\mathbb{Z}}
\def\N{\mathbb{N}}

\def\e{\varepsilon}

\newcommand{\lin}{\mathrm{lin}}

\numberwithin{equation}{section}

\begin{document}

\title[A Blichfeldt inequality for centrally symmetric convex bodies]{A Blichfeldt-type inequality for centrally symmetric convex bodies}
\author{Matthias Henze}

\address{Fakult\"at f\"ur Mathematik, Otto-von-Guericke
Universit\"at Mag\-deburg, Universit\"atsplatz 2, D-39106 Magdeburg,
Germany}
\email{matthias.henze@ovgu.de}
\thanks{This work was supported by the Deutsche
  Forschungsgemeinschaft (DFG) within the project He 2272/4-1.}

\subjclass[2010]{52C07, 52B20, 52A40, 11H06}

\keywords{Blichfeldt-type inequalities, Davenport inequality, central symmetry, lattice point enumerator}

\begin{abstract}
In this note, we derive an asymptotically sharp upper bound on the number of lattice points in terms of the volume of centrally symmetric convex bodies. Our main tool is a generalization of a result of Davenport that bounds the number of lattice points in terms of volumes of suitable projections.
\end{abstract}

\maketitle

\section{Introduction}

Let $\K^n$ be the set of all convex bodies in $\R^n$, i.e., compact convex sets $K$ with nonempty interior $\inter K$. Such a body $K$ is called {\em centrally symmetric} if $K=-K$. The family of $n$-dimensional lattices in $\R^n$ is denoted by $\L^n$ and the usual Lebesgue measure with respect to the $n$-dimensional Euclidean space by $\vol_n(\cdot)$. If the ambient space is clear from the context, we omit the subscript and just write $\vol(\cdot)$. For a bounded subset $S\subset\R^n$ the lattice point enumerator is denoted by $\LE(S)=\#(S\cap\Z^n)$. A lattice polytope is a polytope all of whose vertices are lattice points in $\Z^n$. Finally, for an $A\subseteq\R^n$ we denote the dimension of its affine hull by $\dim A$.

We are interested in bounds on the volume in terms of the lattice point enumerator of a convex body. For $K\in\K^n$ with $\dim(K\cap\Z^n)=n$ a sharp lower bound on $\vol(K)$ was obtained by Blichfeldt \cite{Blichfeldt1921}, which reads \[\vol(K)\geq\frac1{n!}\left(\LE(K)-n\right).\]
We call results of this kind \emph{Blichfeldt-type inequalities}.
On the other hand, the best known upper bound on $\vol(P)$ for a lattice polytope $P\in\K^n$ is due to Pikhurko \cite{Pikhurko2001}
\[\vol(P)\leq(8n)^n15^{n2^{2n+1}}\LE(\inter P)\] and holds under the condition that $\LE(\inter P)\neq0$.
On the class of centrally symmetric convex bodies, Blichfeldt \cite{Blichfeldt1921} and van der Corput \cite{vdCorput1936} obtained a sharp upper bound on the volume
\begin{equation}
\vol(K)\leq2^{n-1}\left(\LE(\inter K)+1\right)\quad\textrm{for all}\quad K\in\K^n_0.\label{thm_blich_vdC}
\end{equation}

Bey, Henk and Wills \cite{BeyHenkWills2007} proposed the study of a reverse inequality also on the class of centrally symmetric convex bodies, and in \cite{HenkHenzeWills2011} the authors derive, as a first step, Blichfeldt-type inequalities for lattice crosspolytopes, lattice zonotopes, and for centrally symmetric planar convex sets. Moreover, they conjecture that there is a constant $c>1$ such that $\vol(K)\geq\frac{c^n}{n!}\LE(K)$ for every $K\in\K^n_0$ with $\dim(K\cap\Z^n)=n$.

In this work, we confirm this conjecture asymptotically by showing that for every $\e\in(0,1]$ and large enough $n\in\N$ a valid choice for this constant is $c=2-\e$. As the main ingredient to our argument, we prove the following generalization of a classical result of Davenport \cite{Davenport1951}. Therein, we denote by $\Q(P)$ the set of all lattice parallelepipeds in $\K^n$ whose edges are parallel to a given lattice parallelepiped $P\in\K^n$. Moreover, $K|L$ denotes the orthogonal projection of $K$ onto the subspace $L$ and $\binom{[n]}{i}$ is the set of all $i$-element subsets of $[n]=\{1,\dots,n\}$.

\begin{lemma}\label{lem_gen_Davenport}
Let $K\in\K^n$ and let $P=\sum_{j=1}^n[0,z_j]$ be a lattice parallelepiped. Then \[\LE(K)\leq\sum_{i=0}^n\sum_{J\in\binom{[n]}{i}}\vol_{n-i}(K|L_J^\perp)\vol_i(P_J),\] where $L_J=\lin\{z_j:j\in J\}$ and $P_J=\sum_{j\in J}[0,z_j]$ for each $J\in\binom{[n]}{i}$. Equality holds if and only if $P$ is a fundamental cell of $\Z^n$ and $K\in\Q(P)$.
\end{lemma}

As mentioned above, we use the preceding result to derive a Blichfeldt-type inequality for centrally symmetric convex bodies and thereby confirm conjectured bounds from \cite[Conj.~1.1]{BeyHenkWills2007} and \cite{HenkHenzeWills2011} asymptotically.

\begin{theorem}\label{thm_appl_gen_Davenport}
For every $\e\in(0,1]$ there exists an $n(\e)\in\N$ such that for every $n\geq n(\e)$ and every $K\in\K^n_0$ with $\dim(K\cap\Z^n)=n$, we have
\[\vol(K)\geq\frac{(2-\e)^n}{n!}\LE(K).\]
The constant $2$ cannot be replaced by a bigger one.
\end{theorem}

As an application of this inequality we bound the magnitude $\frac{\LE(K)}{\LE(K^\star)\vol(K)}$ by constants that depend on the dimension $n$ but not on the body $K$. Recall that $K^\star=\{x\in\R^n:x^\intercal y\leq1,\,\forall y\in K\}$ denotes the polar body of $K\in\K^n_0$. Estimates of such kind were first studied and applied by Gillet and Soul\'{e}~\cite{GilletSoule1991} who obtained $6^{-n}\leq\frac{\LE(K)}{\LE(K^\star)\vol(K)}\leq\frac{6^nn!}{c^n}$ for some absolute constant $c\leq4$.

\begin{theorem}\label{thm_GillSoul}
For every $\e>0$ there exists an $n(\e)\in\N$ such that for every $n\geq n(\e)$ and every $K\in\K^n_0$ with $\dim(K\cap\Z^n)=n$, we have
\[(\pi+\e)^{-n}\leq\frac{\LE(K)}{\LE(K^\star)\vol(K)}\leq\frac{(\pi+\e)^nn!}{c^n},\]
where $c\leq4$ is an absolute constant.
\end{theorem}

The results of Lemma \ref{lem_gen_Davenport} and Theorem \ref{thm_appl_gen_Davenport} are discussed in the subsequent section. In Section \ref{sect_3}, we give the details of the proof of Theorem \ref{thm_GillSoul}.

\section{Proofs of Lemma \ref{lem_gen_Davenport} and Theorem \ref{thm_appl_gen_Davenport}}

A convex body $T\in\K^n$ is said to be a \emph{lattice tile with respect to the lattice $\Lambda\in\L^n$} if $T$ tiles $\R^n$ by vectors in $\Lambda$, that is, $\R^n=\Lambda+T$ and $(x+\inter T)\cap(y+\inter T)=\emptyset$, for all different $x,y\in\Lambda$. It is well-known that lattice tiles are polytopes, and thus we can assume that every lattice tile has the origin as a vertex. For a survey on tilings and references to the relevant literature, we refer the reader to~\cite{Schulte1993}.

Betke and Wills~\cite{BetkeWills1979} (cf.~\cite[Sect.~3]{GritzWills1993}) showed that, for every convex body $K\in\K^n$, the number of lattice points in $K$ is bounded by $\LE(K)\leq\vol(K+L)$, where $L$ is a fundamental cell of $\Z^n$. They asked to determine all bodies $L$ that admit such an inequality. With the following lemma, we identify lattice tiles as bodies with this property.

\begin{lemma}\label{lem_G_vol_Z_inequ}
Let $K\in\K^n$ and let $T$ be a lattice tile with respect to a sublattice $\Lambda$ of $\Z^n$. Then \[\LE(K)\leq\vol(K+T).\]
If $T$ is a lattice parallelepiped $P$, then equality holds if and only if $\Lambda=\Z^n$ and $K\in\Q(P)$.
\end{lemma}
\begin{proof}
Since for all $x,y\in\Lambda$ we have $(x+\inter T)\cap(y+\inter T)=\emptyset$, unless $x=y$, every residue class modulo $\Lambda$ is a packing set of $T$. Let $\{r_1,\dots,r_m\}\subset\Z^n$ be a maximal subset of different representatives of residue classes modulo~$\Lambda$. Writing $\Lambda_j=r_j+\Lambda$, we have for every $j=1,\dots,m$ that \[\#(K\cap\Lambda_j)=\frac{\vol\left((K\cap\Lambda_j)+T\right)}{\vol(T)}.\]
Since $T$ is a lattice tile, we have $\vol(T)=\det\Lambda=m$, and therefore \[\LE(K)=\sum_{j=1}^m\#(K\cap\Lambda_j)=\frac1{m}\sum_{j=1}^m\vol\left((K\cap\Lambda_j)+T\right)\leq\vol(K+T).\]
By the compactness of the involved sets, equality is attained if and only if $(K\cap\Lambda_j)+T=K+T$ for all $j=1,\dots,m$. In particular, there can only be one residue class and thus $m=\det\Lambda=1$, which means $\Lambda=\Z^n$. In the case that the lattice tile is a lattice parallelepiped $P=\sum_{i=1}^n[0,a_i]$, every hyperplane supporting a facet of the convex polytope $K+P=(K\cap\Z^n)+P$ is parallel to a hyperplane supporting a facet of $P$. Therefore, $K+P$ is a lattice translate of $\sum_{i=1}^n[0,t_ia_i]$ for some $t_i\in\N$, and so $K$ is a lattice translate of $\sum_{i=1}^n[0,(t_i-1)a_i]\in\Q(P)$.

Conversely, if $P$ is a fundamental cell of $\Z^n$, then we find lattice vectors $v_1,\ldots,v_n\in\Z^n$ such that, up to a lattice translation, $P=\sum_{i=1}^n[0,v_i]$. Again, up to a lattice translation, every $K\in\Q(P)$ is of the form $K=\sum_{i=1}^n[0,l_iv_i]$ for some $l_1,\ldots,l_n\in\N$. Since $P$ is a fundamental cell, we have $\vol(P)=1=\#\big(\sum_{i=1}^n[0,v_i)\cap\Z^n\big)$ and thus
\[\LE(K)=\#\left(\sum_{i=1}^n[0,(l_i+1)v_i)\cap\Z^n\right)=\vol(K+P).\qedhere\]
\end{proof}

\begin{proof}[Proof of Lemma \ref{lem_gen_Davenport}]
The lattice parallelepiped $P=\sum_{j=1}^n[0,z_j]$ is clearly a lattice tile with respect to the sublattice of $\Z^n$ which is spanned by $z_1,\dots,z_n$. Based on an alternative proof by Ulrich Betke of an inequality of Davenport \cite{Davenport1951} we use Lemma \ref{lem_G_vol_Z_inequ} and develop the volume of $K+P$ into a sum of the mixed volumes $\iV(K,n-i;P,i)$ of $K$ and $P$ (we refer to the books of Gardner~\cite{Gardner1995} and Schneider \cite{Schneider1993} for details and properties on mixed volumes)
\begin{eqnarray}
\LE(K)&\leq&\vol(K+P)=\sum_{i=0}^n\binom{n}{i}\iV(K,n-i;P,i).\label{eqn_1}
\end{eqnarray}
By the linearity and nonnegativity of the mixed volumes we have
\begin{eqnarray}
\iV(K,n-i;P,i)&=&\sum_{j_1=1}^n\dots\sum_{j_i=1}^n\iV(K,n-i;[0,z_{j_1}],\dots,[0,z_{j_i}])\nonumber\\
&=&\sum_{J\in\binom{[n]}{i}}i!\iV(K,n-i;[0,z_j],j\in J),\label{eqn_2}
\end{eqnarray}
and by Equation (A.41) in \cite[App.~A.5]{Gardner1995} it holds
\begin{eqnarray}
\binom{n}{i}i!\iV(K,n-i;[0,z_j],j\in J)&=&\vol_{n-i}(K|L_J^\perp)\vol_i(P_J),\label{eqn_3}
\end{eqnarray}
for any $J\in\binom{[n]}{i}$. Finally, combining (\ref{eqn_1}), (\ref{eqn_2}) and (\ref{eqn_3}) gives the desired result.\par
The equality characterization is inherited from Lemma \ref{lem_G_vol_Z_inequ} since (\ref{eqn_1}) is the only step where there could be an inequality.
\end{proof}

For later reference, we state the following lemma that can be found for example in \cite[Thm.~1]{RogersShephard1958} and \cite[Lem.~3.1]{BourgainMilman1987}.

\begin{lemma}\label{lem_volume_inequ}
Let $K\in\K^n_0$ and let $L$ be an $i$-dimensional linear subspace of $\R^n$. Then \[\vol(K)\leq\vol_{n-i}(K|L^\perp)\vol_i(K\cap L)\leq\binom{n}{i}\vol(K),\] and both inequalities are best possible.
\end{lemma}

\begin{proof}[Proof of Theorem \ref{thm_appl_gen_Davenport}]
By assumption, we find $n$ linearly independent lattice points $z_1,\dots,z_n$ inside $K$. Applying Lemma \ref{lem_gen_Davenport} with respect to the lattice parallelepiped $P=\sum_{j=1}^n[0,z_j]$ gives
\[\LE(K)\leq\sum_{i=0}^n\sum_{J\in\binom{[n]}{i}}\vol_{n-i}(K|L_J^\perp)\vol_i(P_J).\]
By the construction of the subspaces $L_J$, we have
\[\vol_i(K\cap L_J)\geq\vol_i\big(\conv\{\pm z_j:j\in J\}\big)=\frac{2^i}{i!}\vol_i(P_J),\]
and together with Lemma \ref{lem_volume_inequ} we get
\begin{eqnarray}
\LE(K)&\leq&\sum_{i=0}^n\sum_{J\in\binom{[n]}{i}}\frac{i!}{2^i}\vol_{n-i}(K|L_J^\perp)\vol_i(K\cap L_J)\nonumber\\
&\leq&\vol(K)\sum_{i=0}^n\binom{n}{i}^2\frac{i!}{2^i}=\vol(K)\frac{n!}{2^n}L_n(2),\label{eqn_ch3_exact_sym_blich}
\end{eqnarray}
where $L_n(x)=\sum_{k=0}^n\binom{n}{k}\frac{x^k}{k!}$ denotes the $n$th Laguerre polynomial\index{polynomial!Laguerre}. For two functions $f,g:\N\to\R$, we denote by $f(n)\approx g(n)$ that $\lim_{n\to\infty}\frac{f(n)}{g(n)}=1$. In Szeg{\H{o}}'s book~\cite[p.~199]{Szego1975} one finds the approximation
\[L_n(x)\approx\frac{n^{-\frac14}}{2\sqrt{\pi}}\frac{e^{-\frac{x}2}}{x^{\frac14}}e^{2\sqrt{x(n+\frac12)}}\quad\textrm{ for all fixed }\quad x>0.\]
Therefore, by $\lim_{n\to\infty}e^{\frac{2\sqrt{2n+1}}{n}}=1$, we have
\begin{eqnarray}
\frac{L_n(2)}{2^n}&\approx&\frac1{2e\sqrt{\pi}\sqrt[4]{2n}}\frac{e^{2\sqrt{2n+1}}}{2^n}<\frac{e^{2\sqrt{2n+1}}}{2^n}\leq\frac1{(2-\e)^n}\label{eqn_ch3_lag_approx}
\end{eqnarray}
for every $\e\in(0,1]$ and large enough $n\in\N$. Hence, for large enough $n$, we arrive at
\[\LE(K)\leq\vol(K)\frac{n!}{2^n}L_n(2)\leq\vol(K)\frac{n!}{(2-\e)^n}.\]

In order to see that this inequality is asymptotically sharp, we consider the crosspolytope $C_{n,l}^\star=\conv\{\pm le_1,\pm e_2,\dots,\pm e_n\}$. We have $\LE(C_{n,l}^\star)=2(n+l)-1$ and $\vol(C_{n,l}^\star)=\frac{2^n}{n!}l$. Therefore, $\sqrt[n]{\frac{n!\vol(C_{n,l}^\star)}{\LE(C_{n,l}^\star)}}=\sqrt[n]{\frac{2^nl}{2(n+l)-1}}$ tends to $2$ when $l$ and $n$ tend to infinity. On the other hand, the above inequality shows that for every $\e\in(0,1]$ we have $\sqrt[n]{\frac{n!\vol(K)}{\LE(K)}}\geq 2-\e$ for $n\to\infty$.
\end{proof}

\section{An application of Theorem \ref{thm_appl_gen_Davenport}}\label{sect_3}

In this section, we give the details of the proof of Theorem \ref{thm_GillSoul}. We proceed by showing that Theorem \ref{thm_appl_gen_Davenport} can be used to derive an improvement of \[\LE(K)\vol(K^\star)\leq6^n.\]
This inequality is due to Gillet and Soul\'{e} \cite{GilletSoule1991} and holds for every $K\in\K^n_0$ with $\dim(K\cap\Z^n)=n$.
The volume of the Euclidean unit ball $B_n$ is denoted by
\[\kappa_n=\vol(B_n)=\frac{\pi^{\frac{n}2}}{\Gamma(\frac{n}2+1)},\]
where $\Gamma(z)=\int_0^\infty e^{-t}t^{z-1}dt$ is the gamma function (cf.~\cite[p.~13]{Gardner1995}).

\begin{corollary}\label{cor_Gvol_bound}
For every $\e>0$ there exists an $n(\e)\in\N$ such that for every $n\geq n(\e)$ and every $K\in\K^n_0$ with $\dim(K\cap\Z^n)=n$, we have
\[\frac{c^n}{n!}\leq\LE(K)\vol(K^\star)\leq(\pi+\e)^n.\]
Here, $c\leq2$ is an absolute constant and the lower bound holds for every $n\in\N$ and arbitrary $K\in\K^n_0$.
\end{corollary}
\begin{proof}
For the lower bound we combine Inequality (\ref{thm_blich_vdC}) and the estimate
\begin{eqnarray}
\vol(K)\vol(K^\star)&\geq&\frac{C^n}{n!},\label{ineq_BM}
\end{eqnarray}
which is due to Bourgain and Milman \cite{BourgainMilman1987} and holds for some universal constant $C\leq4$. Indeed, we have \[\LE(K)\vol(K^\star)\geq\frac1{2^n}\vol(K)\vol(K^\star)\geq\frac{(C/2)^n}{n!}.\]
Now we restrict to $K\in\K^n_0$ with $\dim(K\cap\Z^n)=n$. The exact inequality~\eqref{eqn_ch3_exact_sym_blich} in the proof of Theorem \ref{thm_appl_gen_Davenport} and the Blaschke-Santal\'{o}~\cite{Santalo1949} inequality imply that
\begin{eqnarray}
\LE(K)\vol(K^\star)&\leq&\frac{n!\,L_n(2)}{2^n}\vol(K)\vol(K^\star)\leq\frac{n!\,\kappa_n^2\,L_n(2)}{2^n}.\label{eqn_ch3_GS_exact}
\end{eqnarray}
Hence, by~\eqref{eqn_ch3_lag_approx}, we have that for every $\e'\in(0,1]$ and large enough $n\in\N$
\[\LE(K)\vol(K^\star)\leq\frac{n!\,\kappa_n^2}{(2-\e')^n}.\]
Stirling approximation and $\kappa_n=\frac{\pi^\frac{n}{2}}{\Gamma(\frac{n}2+1)}$ give
\[\frac{n!\,\kappa_n^2}{(2-\e')^n}\approx\frac{\sqrt{2\pi n}(\frac{n}{e})^n\pi^n}{(2-\e')^n\Gamma(\frac{n}2+1)^2}\approx\frac{\sqrt{2\pi n}(\frac{n}{e})^n\pi^n}{(2-\e')^n\pi n(\frac{n}{2e})^n}\leq\left(\frac{2\pi}{2-\e'}\right)^n.\]
For every $\e\in(0,1)$ there exists an $\e'\in(0,1)$ such that $\frac{2\pi}{2-\e'}\leq\pi+\e$, and we conclude that for large $n$, the inequality $\LE(K)\vol(K^\star)\leq(\pi+\e)^n$ holds.
\end{proof}

In the planar case, we are able to give sharp bounds. Two sets are called unimodularly equivalent if there is a lattice preserving affine transformation that maps one onto the other. With $e_i$ we denote the $i$th unit vector in $\R^n$.

\begin{proposition}
Let $K\in\K^2_0$ be such that $\dim(K\cap\Z^2)=2$. Then \[2\leq\LE(K)\vol(K^\star)\leq21.\] In the upper bound, equality holds if and only if $K$ is unimodularly equivalent to the hexagon $H=\conv\{\pm e_1,\pm e_2,\pm(e_1+e_2)\}$. The lower bound is also best possible and holds for arbitrary $K\in\K^2_0$.
\end{proposition}
\begin{proof}
For the upper bound we can restrict to lattice polygons $P\in\P^2_0$ since for $P_K=\conv\{K\cap\Z^2\}$ we clearly have $\LE(K)\vol(K^\star)\leq\LE(P_K)\vol(P_K^\star)$. By the well-known formula of Pick \cite{Pick1899} we get
\begin{eqnarray*}
\LE(P)\vol(P^\star)&=&\left(\vol(P)+\frac12\LE(\partial P)+1\right)\vol(P^\star)\\
&=&\left(\vol(P)+\frac12\LE(P)-\frac12\LE(\inter P)+1\right)\vol(P^\star)
\end{eqnarray*}
and therefore \[\LE(P)\vol(P^\star)=2\vol(P)\vol(P^\star)+\vol(P^\star)\left(2-\LE(\inter P)\right).\]
Using the Blaschke-Santal\'{o} inequality \cite{Santalo1949} in the plane gives $\LE(P)\vol(P^\star)\leq2\vol(P)\vol(P^\star)\leq2\pi^2<21$, whenever $\LE(\inter P)>1$. Up to unimodular equivalence there are only three centrally symmetric lattice polygons with exactly one interior lattice point: the square $[-1,1]^2$, the diamond $\conv\{\pm e_1,\pm e_2\}$ and the hexagon $H$ (see for example \cite[Prop.~2.1]{Nill2005}). Among these three the hexagon is the only maximizer of $\LE(P)\vol(P^\star)$.\par
For the lower bound we use Mahler's inequality $\vol(K)\vol(K^\star)\geq8$, for $K\in\K^2_0$ (see \cite{Mahler1938}). By the same lines as in the proof of Corollary \ref{cor_Gvol_bound} this yields $\LE(K)\vol(K^\star)\geq2$ which is best possible as shown by the squares $[-1+\varepsilon,1-\varepsilon]^2$, for small $\varepsilon>0$.
\end{proof}

\begin{remark}
Based on computer experiments in small dimensions, we conjecture that the hexagon in the plane is an exception and that for $n\geq3$ the maximizing example is the standard crosspolytope $C_n^\star=\conv\{\pm e_1,\ldots,\pm e_n\}$, i.e.~$\LE(K)\vol(K^\star)\leq(2n+1)2^n$ for every $K\in\K^n_0$ with $\dim(K\cap\Z^n)=n$.
\end{remark}

Finally we are ready to prove Theorem \ref{thm_GillSoul}. The arguments are similar to those given by Gillet and Soul\'{e} \cite[Sect.~1.6]{GilletSoule1991}. We repeat them here with the necessary adjustments to our bounds and for the sake of completeness. 

\begin{proof}[Proof of Theorem \ref{thm_GillSoul}]
The upper bound can be derived by applying the lower bound to $K^\star$ and by using Inequality~\eqref{ineq_BM} of Bourgain and Milman. In fact, there is an absolute constant $c\leq4$ with
\[\frac{\LE(K)}{\LE(K^\star)\vol(K)}=\frac{\LE(K)\vol(K^\star)}{\LE(K^\star)\vol(K)\vol(K^\star)}\leq\frac{(\pi+\e)^n}{\vol(K)\vol(K^\star)}\leq\frac{(\pi+\e)^nn!}{c^n}.\]
Here, $\e>0$ and $n\in\N$ is large enough.

For the lower bound, let $L=\lin(K^\star\cap\Z^n)$ and $k=\dim L$. Let us abbreviate $L_n=L_n(2)$. We use the exact inequality~\eqref{eqn_ch3_GS_exact} from the proof of Corollary~\ref{cor_Gvol_bound} in the sublattice $\Z^n\cap L$ and obtain
\begin{eqnarray}
\LE(K^\star)&=&\LE(K^\star\cap L)\leq\frac{k!\,\kappa_k^2\,L_k}{2^k}\frac{\det((\Z^n\cap L)^\star)}{\vol_k((K^\star\cap L)^\star)}\nonumber\\
&=&\frac{k!\,\kappa_k^2\,L_k}{2^k}\frac{\det(\Z^n|L)}{\vol_k(K|L)}.\label{eqn_ch3_GS_1}
\end{eqnarray}
By $K=(K^\star)^\star$, we have $\inter K=\{x\in K:|x^\intercal y|<1,\,\forall y\in K^\star\}$ and thus $x^\intercal y=0$ for all $x\in\inter K\cap\Z^n$ and $y\in K^\star\cap\Z^n$. Therefore, using Theorem~\ref{thm_blich_vdC} in the sublattice $\Z^n\cap L^\perp$, we get
\begin{eqnarray}
\LE(K)&\geq&\LE(\inter K)=\LE(\inter K\cap L^\perp)\nonumber\\
&\geq&\frac1{2^{n-k}}\frac{\vol_{n-k}(\inter K\cap L^\perp)}{\det(\Z^n\cap L^\perp)}=\frac1{2^{n-k}}\frac{\vol_{n-k}(K\cap L^\perp)}{\det(\Z^n\cap L^\perp)}.\label{eqn_ch3_GS_2}
\end{eqnarray}
Combining~\eqref{eqn_ch3_GS_1}, \eqref{eqn_ch3_GS_2} and the lower bound in Lemma~\ref{lem_volume_inequ} gives \[\frac{\LE(K)}{\LE(K^\star)}\geq\frac{4^k}{2^n\,k!\,\kappa_k^2\,L_k}\frac{\vol_{n-k}(K\cap L^\perp)\vol_k(K|L)}{\det(\Z^n\cap L^\perp)\det(\Z^n|L)}\geq\frac{4^k\vol(K)}{2^n\,k!\,\kappa_k^2\,L_k}.\]
What remains is to show that the function $g(k)=\frac{4^k}{k!\,\kappa_k^2\,L_k}$ is nonincreasing for $k\geq0$, because together with the previous inequality we then arrive at
\[\frac{\LE(K)}{\LE(K^\star)\vol(K)}\geq\frac{4^k}{2^n\,k!\,\kappa_k^2\,L_k}\geq\frac{2^n}{n!\,\kappa_n^2\,L_n}.\]
In the proof of Corollary~\ref{cor_Gvol_bound}, we have seen that for every $\e>0$ the right hand side of the above inequality is at least $(\pi+\e)^{-n}$ for large enough $n\in\N$.

Now $g(k)\geq g(k+1)$ if and only if $\frac{k+1}{4}\geq\frac{\kappa_k^2}{\kappa_{k+1}^2}\frac{L_k}{L_{k+1}}$. By the estimate $\frac{\kappa_k^2}{\kappa_{k+1}^2}\leq\frac{k+2}{2\pi}$ (see~\cite[Lem.~1]{BetkeGritzWills1982}), it is therefore enough to prove $\frac{k+1}4\geq\frac{k+2}{2\pi}\frac{L_k}{L_{k+1}}$. After an elementary calculation, we see that this follows from the recurrence relation $L_{k+1}=\sum_{i=0}^{k+1}\binom{k+1}{i}\frac{2^i}{i!}=2L_k-\frac{2^k(k-1)}{(k+1)!}$.
\end{proof}

\medskip
\noindent
{\it Acknowledgment.} We thank Martin Henk for many valuable comments and suggestions.

\providecommand{\bysame}{\leavevmode\hbox to3em{\hrulefill}\thinspace}
\providecommand{\MR}{\relax\ifhmode\unskip\space\fi MR }
\providecommand{\MRhref}[2]{%
  \href{http://www.ams.org/mathscinet-getitem?mr=#1}{#2}
}
\providecommand{\href}[2]{#2}

\end{document}